\documentclass{article}

\usepackage{graphicx}
\usepackage{xcolor}
\usepackage{float}
\usepackage{amssymb}
\usepackage{amsmath}
\usepackage{mathrsfs}
\usepackage{bm}
\usepackage{amsthm}

\newcommand{\Real}{\mathbb{R}}
\newcommand{\Complex}{\mathbb{C}}

\newcommand{\abs}[1]{\left\vert#1\right\vert}
\newcommand{\set}[1]{\left\{#1\right\}}

\newtheorem{thm}{Theorem}[subsection]

\newtheorem{cor}[thm]{Corollary}
\newtheorem{lem}[thm]{Lemma}
\newtheorem*{lemma}{Lemma}
\newtheorem{prop}[thm]{Proposition}
\newtheorem*{proposition}{Proposition}
\newtheorem*{example}{Example}
\theoremstyle{definition}
\newtheorem{defn}[thm]{Definition}

\theoremstyle{remark}
\newtheorem{rem}[thm]{Remark}
\numberwithin{equation}{subsection}
\title{Parabolic metrics with conical singularities on compact Riemann surfaces}

\author{Santai Qu}

\begin{document}

\maketitle
\begin{abstract}
  In this paper, we prove the existence and uniqueness theorem for parabolic conical metrics on Riemann surfaces in the situation of generalized real angles, positive, zero and negative, by complex analysis, and give an example of this theorem to clarify concrete expressions of parabolic metrics on the two-sphere and generalize the well-known Schwarz-Christoffel formula.
\end{abstract}
\section{Introduction}
The goal of this paper is to prove the existence and uniqueness theorem of parabolic metrics (i.e. constant curvature zero) with conical singularities of arbitrary real angles, not only positive angles, on Riemann surfaces.
\begin{defn}
Let $M$ be a compact Riemann surface.  A conical metric $g$ on $M$ with a \emph{conical singularity} at $P\in M$ with \emph{singular angle} $2\pi \alpha>0$ is a conformal metric on $M$ which could be expressed locally as
\[g=e^{2u}\abs{dz}^2,\]
where $z$ is a local coordinate near $P$ such that $z(P)=0$.  Moreover, the real-valued function $u$ satisfies that
\[u(z)-(\alpha-1)\cdot\log{\abs{z}}\]
is continuous at $P$.  And we call that the metric $g$ represents the divisor 
\[D=\prod_{j=1}^n P_j^{\alpha_j-1}.\]
\end{defn}

This is the classical definition.  In this paper, we only consider the case of constant curvature $K=0$ and extend the above definition to arbitrary real angles, not only positive.

By the work of \cite{McOwen 1988} and \cite{Troyanov 1991}, it is known that for constant curvature $K\leq 0$ and positive conical angles such metric exists uniquely if and only if
\[\int_M K=2\pi(\chi(M)+deg D),\]
where $deg D=\sum_{j=1}^n(\alpha_j-1)$ and $\chi(M)$ is the Euler characteristic of $M$.  This is the well-known Gauss-Bonnet condition.  And in our situation, we shall prove the following existence and uniqueness theorem.
\begin{thm}{\bf(Main Theorem)}
If $M$ is a compact Riemann surface and $D=\prod_jP_j^{(\alpha_j-1)}$ is an arbitrary given divisor with real coefficients,
then there exists a unique parabolic metric $g$ (up to multiplying nonzero constants) representing the divisor $D$ if and only if the Gauss-Bonnet condition holds, that is, 
\[\sum_{j=1}^n(\alpha_j-1)=\chi(M).\]
\end{thm}

The main tools to prove this theorem are developing maps and Prym differentials.  The idea of developing map is introduced by R. Bryant \cite{Bryant 1988}, Umehara-Yamada \cite{UY00} and Eremenko \cite{Er04}.
\begin{defn}
(Umehara-Yamada\cite{UY00}) Let $g$ be a conformal metric on $M$ of constant curvature one and represent the divisor $D$.  We call a multi-valued locally univalent meromorphic function $\bm{F}$ on $\Sigma=M\setminus\{P_1, \cdot\cdot\cdot, P_n\}$ a \emph{developing map} of the metric $g$ if $g=\bm{F}^*g_{st}$, where $g_{st}$ is the standard metric of curvature one on Riemann sphere as
\[g_{st}=\dfrac{4\abs{z}^2\abs{dz}^2}{(1+\abs{z}^2)^2}.\]
\end{defn}
In our situation, developing maps are defined as following.
\begin{defn}
If $g$ is a conformal metric with constant curvature zero on a compact Riemann surface  and represents the divisor $D$, a \emph{developing map} of $g$ is a multi-valued locally univalent meromorphic function $\bm{F}$ mapping $\Sigma=M\setminus\{P_1, \cdot\cdot\cdot, P_n\}$ into $\mathbb{C}$ satisfying $g=\bm{F}^*(\abs{dz}^2)$, where $\abs{dz}^2$ is the standard metric on complex plane $\mathbb{C}$.
\end{defn}
The existence of developing maps is proved in \cite{Binxu}.  And by the Remark 2.1 in \cite{Binxu}, we know that developing maps also exist for parabolic and hyperbolic (constant negative curvature) conformal metrics with finite conical or cusp singularities.  

  For Prym differentials, we make the following definitions.
\begin{defn}
Let $\Sigma$ be a connected Riemann surface, which need not to be compact, and its fundamental group is denoted by $\pi_1(\Sigma)$.  By a character $\chi$, we mean a group homomorphism
\[\chi: \pi_1(\Sigma)\rightarrow\mathbb{C}^*,\]
which could descend to
\[\chi: H_1(\Sigma)\rightarrow\mathbb{C}^*.\]
A character is said to be \emph{normalized} if the norms of its images are always one.

And, by a multiplicative multi-valued function $\bm{f}$ belonging to the character $\chi$, we mean a collection of function elements, say, $\set{(f,U)}$, where $U$ is open in $\Sigma$.  Moreover, the collection should satisfy:
\begin{itemize}
\item{for any two given elements $(f_1,U_1)$ and $(f_2,U_2)$ in $\bm{f}$, then $(f_2,U_2)$ could be obtained by analytic continuation of $(f_1,U_1)$ along some curve $c$ on $\Sigma$, and}
\item{continuation of a function element $(f,U)$ in $\bm{f}$ along the closed curve $c$ leads to the function element $(\chi(c)f,V)$.}
\end{itemize}
\end{defn}

\begin{defn}
By a \emph{multiplicative differential} or \emph{Prym differential}, we  mean a collection, say $\bm{\omega}$, of triples $(\omega,U,z)$ where $U$ is an open set in $\Sigma$, $z$ is a local coordinate on $U$, and $\omega$ is a holomorphic function of $z$.  Moreover, the triples in $\bm{\omega}$ should behavior compatibly on their intersections.  So if the intersection of two triples, $(\omega,U,z)$ and $(\omega_1,V,\zeta)$, is not empty, then they satisfy that
\[\omega_1(\zeta)d\zeta=\omega(z)dz.\]
\end{defn}

Given two such triples $(\omega_0,U_0,z_0)$ and $(\omega_n,U_n,z_n)$, we say that they are continuation for each other if there is a chain of triples overlapping compatibly, which means we have
\[(\omega_j,U_j,z_j),\,\,\,\,\,\,\,j=0,\cdot\cdot\cdot,n,\]
and
\[\omega_j(z_j)dz_j=\omega_{j+1}(z_{j+1})dz_{j+1},\,\,z_{j+1}=z_{j+1}(Q),\,\,\,\,\,\,Q\in U_j\cap U_{j+1}.\]

By a \emph{multiplicative differential belonging to a character $\chi$}, we mean a multiplicative differential $\bm{\omega}$ and if we continue an element $(\omega,U,z)$ along a closed curve $c$ to the element $(\omega_1,U_1,z)$, then
\[\omega_1(z)=\chi(c)\omega(z),\,\,\,\,\,\,z=z(Q), \,\,\,\,\,Q\in U\cap U_1.\]

The concepts of regularity and divisors with real coefficients for multiplicative functions and Prym differentials will be defined in section 2. 

The uniqueness of parabolic metrics would be proved via the existence and uniqueness theorem of regular multiplicative differentials for a given real divisor and the theorem giving local expressions for developing maps near conical singularities.  And the existence of parabolic metrics will be proved also by the existence and uniqueness theorem of regular multiplicative differentials.

The main idea of the proof is that the differential of a developing map is a regular Prym differential, and conversely developing maps could be obtained by integrating a regular Prym differential.  Moreover, by the definition of developing maps, a parabolic metric on $M$ could be obtained by pulling back the standard metric on $\mathbb{C}$.  Then from the uniquely existence of regular Prym differentials of a given real divisor, we could deduce the existence and uniqueness of developing maps, and hence the uniquely existence of parabolic metrics for a given real divisor.

This proving process is achieved by the following technical theorems.
\begin{thm}
Every divisor $D$ with real coefficients of degree zero, having the form $D=P_1^{\alpha_1}\cdot\cdot\cdot P_r^{\alpha_r}/Q_1^{\alpha_1}\cdot\cdot\cdot Q_r^{\alpha_r}$ with $P_j\neq Q_k$ for all $j, k=1,\ldots, r>0$, is the divisor of a unique (up to a multiplicative constant) regular multiplicative function defined on $\Sigma=M\setminus D$ and belongs to a unique normalized character.
\end{thm}

\begin{thm}
Every divisor $D$ with real coefficients of degree $\chi(M)$, the Euler characteristic of the Riemann surface $M$, is the divisor of a unique (up to a multiplicative constant) regular multiplicative differential defined on $\Sigma=M\setminus D$ and belonging to a unique normalized character.
\end{thm}

\begin{thm}
Let $\bm{F}$ be a developing map of parabolic metric $g$ representing the divisor $D=\prod_jP_j^{(\alpha_j-1)}$.  Then there always exist local coordinate $x$ near a conical singularity such that $\bm{F}$ could be expressed locally as
\begin{eqnarray*}
  \bm{F}(x) &=& C_0x^{\alpha_j}+C_1\,\,\,\,\,\,\,\,\,\,if\,\,\,\,\alpha_j\neq0; \\
  \bm{F}(x) &=& C_0^{\prime}\log{x}+C_1^{\prime} \,\,\,\,\,\,\,\,\,\,if\,\,\,\,\alpha_j=0,
\end{eqnarray*}
where $C_0$ and $C_1$ are constants with $C_0\neq0$.
\end{thm}

Finally, after proving the main geometrical theorem, we give concrete expressions of parabolic metrics on the two-sphere $\mathbb{S}^2$ as an example.  And the developing map in this case will generalize the well-known Schwarz-Christoffel formula.

\section{Preparations for proofs}

\subsection{Integrating a multiplicative differential along arcs}
Let $\gamma:[0,1]\rightarrow M$ be a piecewise smooth arc.  Suppose a multiplicative differential $\bm{\omega}$ can be defined on the image of $\gamma$.  So there exists triple $(\omega_0,U_0,z_0)$ such that $\gamma(0)\in U_0$, where $(\omega_0,U_0,z_0)$ is a branch of $\bm{\omega}$.  Since $\bm{\omega}$ can be defined on the compact image of $\gamma$, we can find a continuation from $(\omega_0,U_0,z_0)$ as a finite chain of triples
\[\{(\omega_j,U_j,z_j)\}_{j=1}^N,\]
with
\[\gamma([0,1])\subset\bigcup_{j=1}^N U_j.\]

Divide the interval into $0=t_0<t_1<\cdot\cdot\cdot<t_N=1$, where $t_{j+1}$ is in the intersection of $U_j$ and $U_{j+1}$.  Let
\[\int_{\gamma} \omega\triangleq\sum_{j=0}^{N-1} \int_{t_j}^{t_{j+1}}\omega_j(\gamma(t))\gamma'(t)dt,\]
where the integral at the right hand side is the usual Lebesgue's integration.  Note that subdivision is required at the discontinuous point of $\gamma$.
\begin{rem}
By
\[\omega_1(\zeta(z))d\zeta=\omega(z)dz,\]
the value of the integral in the definition can be calculated as
\[\omega_1(\zeta(\gamma(t)))\gamma'(t)dt\]
\[=\omega_1(\zeta(\gamma(t)))d(\zeta(\gamma(t)))\]
\[=\omega(z(\gamma(t)))d(z(\gamma(t)))\]
\[=\omega(z(\gamma(t)))\gamma'(t)dt.\]
Therefore, the value of integral is independent of the choice of coordinate.  And a moment thought shows that the integral is also independent of the choice of points of division, namely, $0=t_0<t_1<\cdot\cdot\cdot<t_N=1$.
\end{rem}
\begin{rem}
A prescribed triple $(\omega_0,U_0,z_0)$ is required before doing the integration.
\end{rem}

\subsection{Definition of divisors with real coefficients}
The following definitions are also quoted from \cite{riesurf}.  But we generalize these definitions to divisors with coefficients in $\Real$.

\begin{defn}
A \emph{divisor} on $\Sigma$, which need not to be compact, is a formal symbol
\[\mathcal{U}=P^{\alpha_1}_1P^{\alpha_2}_2\cdot\cdot\cdot P^{\alpha_k}_k,\]
with $P_j\in \Sigma$, $\alpha_j\in \Real$.  It also always be written as
\[\mathcal{U}=\prod_{P\in \Sigma}P^{\alpha(P)},\]
with $\alpha(P)\neq0$ for only finitely many $P\in \Sigma$.
\end{defn}
The free commutative group is call the \emph{group of divisors} on $\Sigma$ and we denote it by $Div(\Sigma)$.  The group operations are defined as following.  For another divisor
\[\mathcal{B}=\prod_{P\in \Sigma}P^{\beta(P)},\]
we write
\begin{eqnarray*}
  \mathcal{UB} &=& \prod_{P\in \Sigma}P^{\alpha(P)+\beta(P)} \\
  \mathcal{U}^{-1} &=& \prod_{P\in \Sigma}P^{-\alpha(P)}.
\end{eqnarray*}
And we write the unit element in $Div{\Sigma}$ as $1$.

And for $\mathcal{U}\in Div(\Sigma)$ given above, we define
\[deg\,\,\mathcal{U}=\sum_{P\in \Sigma}\alpha(P).\]

\subsection{Definitions for the Riemann surfaces $\Sigma=M\setminus D$}

In our context, the Riemann surface that we consider is $\Sigma=M\setminus D$, where $D$ is the divisor $\prod_{j=1}^nP_j^{(\alpha_j-1)}$.  In this paper, we only consider a particular class of multiplicative functions having \emph{regular singularities} at $D$, meaning the single-valued function $\frac{d\bm{f}}{\bm{f}}$ has simple poles at every point of $D$.
The divisor of $\bm{f}$, denoted by $(\bm{f})$, is defined to be
\[(\bm{f})=\prod_{P\in M}P^{\alpha(P)},\]
\[\alpha(P)=res_P \frac{d\bm{f}}{\bm{f}}.\]
Similarly, we only focus on a particular class of multiplicative differentials on $M\setminus D$ giving the divisor $D$.  This is a little complicated, and we list the definition and some explanations to clarify the idea of this definition.
\begin{defn}
We call a Prym differential $\bm{\omega}$ is \emph{regular} if it is defined on $\Sigma=M\setminus D$, where $M$ is a compact Riemann surface and $D$ is a divisor, and there exists a multiplicative differential $\bm{\omega_0}$ on $M$, which only has zeros or poles at $D$, such that $\frac{\bm{\omega}}{\bm{\omega_0}}$ is a regular multiplicative function on $\Sigma$.
\end{defn}
Note that the existence of $\bm{\omega_0}$ is guaranteed by a corollary in the page 131 of \cite{riesurf}, saying that every divisor $D$ with integer coefficients of degree $2g-2$ is the divisor of a unique (up to a multiplicative constant) multiplicative differential belonging to a unique normalized character.

Now, we give some explanations about this definition.  Suppose that $\bm{\omega}=f(z)dz$ near $P_j$.  Our main problem is to find a way to define the residue of $\frac{df}{f}$ at $P_j$ for $\bm{\omega}=fdz$ and to reflect the regularity of $\bm{\omega}$ at $P_j$.

First, we only need that there exists an $\bm{\omega_0}$, while this may not cause any problem since every multiplicative differential only having zeros or poles at $D$ could do the same work.  To define the divisor from $\bm{\omega}$ at $P_j$, we first assume that $\bm{\omega}$ locally has the form $fdz$ near $P_j$ to get the basic idea of our definition of the divisor of $\bm{\omega}$.

Now, near $P_j$, $\bm{\omega}=fdz$ and $\bm{\omega_0}=gdz$.  Then, we have
\[\frac{\bm{\omega}}{\bm{\omega_0}}=\frac{f}{g}=h,\]
\[\frac{dh}{h}=\frac{df}{f}-\frac{dg}{g}.\]
Since $g$ is meromorphic, the residue of $g$ at $P_j$ is an integer.  What we want to do is to give the definition of $res_{P_j}\frac{df}{f}$ by
\[res_{P_j}\frac{df}{f}=res_{P_j}\frac{dh}{h}+res_{P_j}\frac{dg}{g}.\]

By our definition, if $\bm{\omega}$ is regular, then $\bm{f}=\frac{\bm{\omega}}{\bm{\omega_0}}$ is a regular multiplicative function for some $\bm{\omega_0}$.  We define the divisor of $\bm{\omega}$ to be
\[(\bm{\omega})=(\bm{\omega_0})\prod_{P\in M}P^{\alpha(P)},\]
where $\alpha(P)=res_P\frac{d\bm{f}}{\bm{f}}$, that is
\[(\bm{\omega})=(\bm{\omega_0})(\bm{f}).\]
Note that $(\bm{\omega_0})$ is defined in the usual sense as in \cite{riesurf}, i.e. by using residues since it has local expressions near singularities.

We need to check that $(\bm{\omega})$ is well-defined.  If there is another multiplicative differential, say $\bm{\omega_1}$, defining the regularity of $\bm{\omega}$, then $\frac{\bm{\omega}}{\bm{\omega_1}}=\bm{g}$ for some $\bm{g}$, a regular multiplicative function.  By the relation
\[\dfrac{d(\frac{\bm{f}}{\bm{g}})}{\frac{\bm{f}}{\bm{g}}}=\frac{d\bm{f}}{\bm{f}}-\frac{d\bm{g}}{\bm{g}},\]
we have
\[(\frac{\bm{f}}{\bm{g}})=\dfrac{(\bm{f})}{(\bm{g})}.\]
Also, by
\[\dfrac{d(\bm{fg})}{\bm{fg}}=\frac{d\bm{f}}{\bm{f}}+\frac{d\bm{g}}{\bm{g}},\]
\[(\bm{fg})=(\bm{f})(\bm{g}).\]
Therefore, we arrive at
\begin{eqnarray*}
\dfrac{(\bm{\omega_0})(\bm{f})}{(\bm{\omega_1})(\bm{g})}&=&\biggl(\dfrac{\bm{\omega_0}}{\bm{\omega_1}}\biggl)\cdot\biggl(\dfrac{\bm{f}}{\bm{g}}\biggl)  \\
&=&\biggl(\dfrac{\bm{\omega_0}}{\bm{\omega_1}}\cdot\dfrac{\bm{f}}{\bm{g}}\biggl)  \\
&=&\biggl(\dfrac{\bm{\omega_0}}{\bm{\omega}}\cdot\dfrac{\bm{\omega}}{\bm{\omega_1}}\cdot\dfrac{\bm{f}}{\bm{g}}\biggl)  \\
&=&\biggl(\dfrac{\bm{\omega_0}}{\bm{\omega}}\biggl)\cdot (\bm{f})\biggl(\dfrac{\bm{\omega}}{\bm{\omega_1}}\biggl)\cdot (\bm{g})^{-1}=1.
\end{eqnarray*}

Thus, $(\bm{\omega})$ is well-defined.

\section{Proofs of main theorems}
The following lemma is a specialization of a theorem in the page 52 of \cite{riesurf}.
\begin{lem}
Let $M$ be a compact Riemann surface.  $P$ and $Q$ are two arbitrary distinct points on $M$.  Then there exists a meromorphic differential on $M$ denoted by $\tau_{PQ}$ such that $P$ and $Q$ are simples poles for $\tau_{PQ}$ and
\[ord_P \tau_{PQ}=ord_Q \tau_{PQ}=-1;\]
\[res_P\tau_{PQ}=1;\]
\[res_Q\tau_{PQ}=-1.\]
\end{lem}
\begin{defn}
If $\bm{f}$ is a multiplicative function without zeros and poles, then $\frac{d\bm{f}}{\bm{f}}=d(\log{\bm{f}})$ is a single-valued holomorphic differential.  The character $\chi$ of $\bm{f}$ is called \emph{inessential} and $\bm{f}$ is called a \emph{unit}.
\end{defn}
The following proposition and corollary are stated in page 130 of \cite{riesurf} in the context of divisors with integer coefficients.  However, the proof of these statements have nothing to do with the coefficients of divisors.  Thus, we could apply them in our situation.
\begin{prop}
If $\chi$ is an arbitrary character, then there exists a unique inessential character $\chi_1$ such that $\chi/\chi_1$ is normalized.
\end{prop}
\begin{cor}
A normalized inessential character is trivial.
\end{cor}

Note that a divisor of a single-valued meromorphic differential on $M$ is also called a \emph{canonical divisor}.

\subsection{Proof of theorem 1.0.7}
\begin{proof}
\emph{Step1.The uniqueness.}\\
Let
\[D=P_1^{\alpha_1}\cdot\cdot\cdot P_r^{\alpha_r}/Q_1^{\alpha_1}\cdot\cdot\cdot Q_r^{\alpha_r}\]
with $P_j\neq Q_k$ for all $j, k=1,\ldots, r>0$.

If $f_j$ is a regular multiplicative function belonging to the normalized character $\chi_j$ $(j=1,2)$, and $(f_j)=D$.

We only need to prove a claim that: \emph{$h=\frac{f_1}{f_2}$ could be extended to a multiplicative function without zeros and poles on $M$}.  First, by a direct computation, 
\[\frac{dh}{h}=\frac{df_1}{f_1}-\frac{df_2}{f_2},\]
which is obviously meromorphic at $P_j$'s since $\frac{df_1}{f_1}$ and $\frac{df_2}{f_2}$ are both meromorphic at $P_j$.  Moreover,
\[res_{P_j}\frac{dh}{h}=res_{P_j}\frac{df_1}{f_1}-res_{P_j}\frac{df_2}{f_2}=\alpha_j-\alpha_j=0.\]
So every $P_j$ is neither zeros or poles for $h$; otherwise, $res_{P_j}\frac{dh}{h}$ should be a nonzero integer.
Thus, $\chi_1\chi_2^{-1}$ is inessential and normalized; hence trivial.  Moreover, since $f_1$ and $f_2$ are regular, $\frac{dh}{h}$ has no pole, and hence holomorphic (regularity is used here).  Then $\frac{dh}{h}$ is a constant.  So $h$ is holomorphic and single-valued on $M$ and hence a complex constant.

\emph{Step2.The existence.}\\
It suffices to prove the existence for the divisor $D=P^{\alpha}/Q^{\alpha}$, with $P,Q\in M$, $P\neq Q$.  Recall the normalized meromorphic differential $\tau_{PQ}$.  Define
\[f(P)=\exp{\int_{P_0}^{P}\alpha\tau_{PQ}}.\]
Then $\frac{df}{f}=\alpha\tau_{PQ}$ and the residues are $\alpha$ at $P$, $-\alpha$ at $Q$, and both $P$ and $Q$ are simple poles of $\frac{df}{f}$.  Hence $f$ is a regular multiplicative function with divisor $D$.

The character to which $f$ belongs is not necessarily normalized.  But proposition 3.0.3 shows how to get around this obstacle since divisor of a meromorphic function without zeros and poles is the identity.
\end{proof}

Note that the regularity is used to prove uniqueness.

\subsection{Proof of the theorem 1.0.8}
\begin{proof}
\emph{Step1. The existence.}\\
Let
\[D=P_1^{\alpha(P_1)}\cdot\cdot\cdot P_{s-1}^{\alpha(P_{s-1})}\cdot P_s^{\alpha(P_s)},\]
with $\sum_{k=1}^s \alpha(P_k)=2g-2$.  Then we have $\alpha(P_s)=2g-2-\sum_{k=1}^{s-1} \alpha(P_k)$ and the divisor turns to be
\[D=P_1^{\alpha(P_1)}\cdot\cdot\cdot P_{s-1}^{\alpha(P_{s-1})}\cdot P_s^{2g-2-\sum_{k=1}^{s-1} \alpha(P_k)}\]
\[=\frac{P_1^{\alpha(P_1)}}{P_s^{\alpha(P_1)}}\cdot\cdot\cdot\frac{P_{s-1}^{\alpha(P_{s-1})}}{P_s^{\alpha(P_{s-1})}}\cdot P_s^{2g-2}\]
\[=\mathcal{U}\cdot P_s^{2g-2},\]
where $\mathcal{U}=\frac{P_1^{\alpha(P_1)}}{P_s^{\alpha(P_1)}}\cdot\cdot\cdot\frac{P_{s-1}^{\alpha(P_{s-1})}}{P_s^{\alpha(P_{s-1})}}$ has the form described in our theorem.

Now, let $Z$ be a canonical divisor, so $deg Z=2g-2$.  Consider $\frac{D}{Z}$, which has the form
\[\frac{D}{Z}=\mathcal{U}\cdot\frac{P_s^{2g-2}}{Z}\]
\[=\mathcal{U}\cdot\mathcal{V}.\]

Next, apply our theorem of regular multiplicative functions to divisors $\mathcal{U}$ and $\mathcal{V}$ respectively, then there are two regular multiplicative functions $\bm{f_1}$ and $\bm{f_2}$ with normalized character such that
\[(\bm{f_1}\cdot\bm{f_2})=\dfrac{D}{Z}.\]
Note that $\bm{f_1}\cdot\bm{f_2}$ is also a regular multiplicative function belonging to a normalized character.

Since $Z$ is a canonical divisor, there is a single-valued meromorphic differential $\bm{\omega_0}$ such that
\[(\bm{\omega_0})=Z.\]
Let
\[\bm{\omega}=\bm{f_1}\cdot\bm{f_2}\cdot\bm{\omega_0},\]
and we can make its character normalized if necessary.  Looking at the local expression of $\bm{\omega}$ near every singular point, it is clear that $\bm{\omega}$ is regular.

\emph{Step2. The uniqueness.}\\
Let $\bm{\omega}$ and $\bm{\widetilde{\omega}}$ be two regular Prym differentials, which are defined on $\Sigma$ and both represent the divisor $D$. The normalized characters of these two Prym differentials are $\chi_1$ and $\chi_2$.  By definition,
\[\dfrac{\bm{\omega}}{\bm{\omega_0}}=f;\]
\[\dfrac{\bm{\widetilde{\omega}}}{\bm{\widetilde{\omega_0}}}=\tilde{f},\]
where $f$ and $\tilde{f}$ are regular multiplicative functions on $\Sigma$; $\bm{\omega_0}$ and $\bm{\widetilde{\omega_0}}$ are two multiplicative differentials with only zeros or poles at the points of $D$.
Take
\[h=\dfrac{\bm{\omega}}{\bm{\widetilde{\omega}}}\]
\[=\dfrac{f\cdot\bm{\omega_0}}{\tilde{f}\cdot\widetilde{\bm{\omega_0}}}=\dfrac{f}{\tilde{f}}\cdot\dfrac{\bm{\omega_0}}{\widetilde{\bm{\omega_0}}}.\]
We know that $\dfrac{f}{\tilde{f}}$ is a regular multiplicative function on $\Sigma$ and $\dfrac{\bm{\omega_0}}{\widetilde{\bm{\omega_0}}}$ is a regular multiplicative function on $M$.  Thus, $h$ is also a regular multiplicative function on $\Sigma$.  Moreover,
\[(h)=\biggl(\dfrac{f}{\tilde{f}}\biggl)\cdot\biggl(\dfrac{\bm{\omega_0}}{\widetilde{\bm{\omega_0}}}\biggl)\]
\[=\dfrac{(f)}{(\tilde{f})}\cdot\dfrac{(\bm{\omega_0})}{(\widetilde{\bm{\omega_0}})}=\dfrac{(\bm{\omega})}{(\widetilde{\bm{\omega}})}=1.\]
Therefore, $\dfrac{dh}{h}$ has residue zero at every point on $M$.  Thus $h$ could be viewed as a holomorphic multiplicative function defined on $M$ without zeros and poles.  Thus, $\chi_1\chi_2^{-1}$ is inessential and normalized; hence trivial.  And by its regularity, $h$ can be extended to a holomorphic single-valued function on $M$, meaning that $h$ is a constant.  Then complete the proof for uniqueness.
\end{proof}

\section{Developing maps}
In the paper \cite{Binxu}, the author has stated a lemma (lemma 2.1) claiming the existence for developing maps of conical metrics with constant curvature one on a compact Riemann surface as following.
\begin{lemma}
Let $g$ be a conformal metric on a compact Riemann surface $M$ of constant curvature one and representing the divisor $D=\prod_{j=1}^nP_j^{(\alpha_j-1)}$ with $\alpha_j>0$.  Then there exists a multi-valued locally univalent holomorphic map from $\Sigma=M\setminus D$ to the Riemann sphere $\overline{\mathbb{C}}$ such that the monodromy of $\bm{F}$ belongs to $PSU(2)$ and
\[g=\bm{F}^*g_{st}.\]
Recall that $g_{st}=\frac{4\abs{dw}^2}{(1+\abs{w}^2)^2}$ is the standard metric over $\overline{\mathbb{C}}$.
\end{lemma}
And from the Remark 2.1 in \cite{Binxu}, we know that developing maps also exist for hyperbolic conformal metrics with finite conical or cusp singularities.  While looking at the proof of the above lemma in \cite{Binxu}, we could conclude that the developing maps for parabolic metrics in our situation of arbitrary real angles also exist and they have monodromies in $Iso(\Complex, \abs{dz}^2)$.

As in \cite{Binxu}, a multi-valued locally univalent meromorphic function $\bm{h}$ on $\Sigma$ is said to be \emph{projective} if any two function elements $h_1$, $h_2$ of $\bm{h}$ near a point $P\in \Sigma$ are related by a fractional linear transformation $T\in PGL(2,\mathbb{C}))$, $h_2=T\circ h_1$.

The lemma 3.2 in \cite{Binxu} also give local expressions of developing maps near conical singularities as following.
\begin{lemma}
Let $\bm{F}:\Sigma\rightarrow\overline{\mathbb{C}}$ be a projective multi-valued locally univalent meromorphic function, and the monodromy of $\bm{F}$ belongs to a maximal compact subgroup of $PGL(2,\mathbb{C})$.  If $\bm{F}$ is compatible with the divisor $D=\prod_jP_j^{(\alpha_j-1)}$,
then there exists a neighborhood $U_j$ of $P_j$ with complex coordinate $z$ and $T_j\in PGL(2,\mathbb{C})$ such that $z(P_j)=0$ and $g_j=T_j\circ\bm{F}$ has the form
\[g_j(z)=z^{\alpha_j},\]
where $0<\alpha_j\neq1$.  Moreover, there exists $T\in PGL(2,\mathbb{C})$ such that the pullback $(T\circ\bm{F})^*g_{st}$ of the standard metric $g_{st}$ by $T\circ\bm{F}$ is a conformal metric of constant curvature one, which represents the divisor $D=\prod_jP_j^{(\alpha_j-1)}$.  In particular, if the monodromy of $\bm{F}$ belongs to $PSU(2)$, then the fractional linear transformation $T$ turns out to be the identity map.
\end{lemma}
Now we prove another version of lemma 3.2 in \cite{Binxu} for parabolic metrics with conical angles belonging to any real numbers.
\subsection{The Differential Equation of Developing Maps}
\begin{proposition}
Suppose $g$ is a parabolic metric representing the divisor 
$D=\prod_jP_j^{(\alpha_j-1)}$ and $\bm{F}: \Sigma\rightarrow \Complex$ is a developing map of $g$.  Then the Schwarzian $\{\bm{F},x\}$ of $\bm{F}$ has the form
\[\{\bm{F},x\}=\dfrac{1-{\alpha_j}^2}{2x^2}+\frac{d_j}{x}+\psi_j(x),\]
where $x$ is a local complex coordinate near $P_j$ with $x(P_j)=0$ and $\psi_j(x)$ is holomorphic near $P_j$.  Moreover, $d_j$ and $\psi_j(x)$ depended on the choice of the local coordinate $x$.
\end{proposition}
\begin{proof}
Near $P_j$, we write
\[g=e^{2u}\abs{dz}^2,\]
where $z(P_j)=0$ and $u(z)=(\alpha_j-1)\log{\abs{z}}+u_1(z)$, in which $u_1(z)$ is continuous at $P_j$ and differentiable otherwise.  The curvature $K$ of $g$ is given by
\[K=-4e^{2u}\dfrac{\partial^2u}{\partial z\partial \bar{z}}.\]
In the paper \cite{Troyanov 1991}, Troyanov defines a \emph{projective connection} compatible with the divisor $D$ as
\[\eta(z)=2\biggl(\dfrac{\partial^2u}{\partial z^2}-\biggl(\dfrac{\partial u}{\partial z}\biggl)^2\biggl)dz^2.\]
Moreover, in the paper \cite{Binxu}, the authors have proved that
\[\{\bm{F},z\}dz^2=\eta(z),\]
where $\bm{F}$ is a developing map of a conformal metric of constant curvature one on a compact Riemann surface and
\[\eta(z)=\biggl(\dfrac{1-\alpha_j^2}{2z^2}+\frac{d_j}{z}+\psi_j\biggl)dz^2.\]
However, all of the above proofs in those papers have nothing to do with the angles and curvature until the step in \cite{Troyanov 1991} which is to prove the differentiability of the function $u_1(z)=u(z)-(\alpha_j-1)\log{\abs{z}}$ at $z=0$.  This is easy in the parabolic case since we know that
\[\dfrac{\partial^2u_1}{\partial z\partial \bar{z}}=-4K\abs{z}^{2(\alpha_j-1)}e^{2u_1},\]
which is stated in the page 301 of \cite{Troyanov 1991}.  By $K=0$, this equation is equivalent to
\[\triangle u_1=0.\]
Therefore, by the removability of singularity of harmonic function and that $u_1$ is continuous at $z=0$, we have $u_1$ is a $C^{\infty}$-function at $z=0$.  So, as in \cite{Troyanov 1991} and \cite{Binxu},
\[\{\bm{F},z\}=2\biggl(\dfrac{\partial^2u_1}{\partial z^2}-\biggl(\dfrac{\partial u_1}{\partial z}\biggl)^2\biggl)-2\frac{\alpha_j-1}{z}\cdot\frac{\partial u_1}{\partial z}-\dfrac{\alpha_j^2-1}{2z^2}.\]
Since
\[\dfrac{\partial}{\partial \bar{z}}\biggl(\dfrac{\partial u_1}{\partial z}\biggl)=0,\]
$\frac{\partial u_1}{\partial z}$ is holomorphic near $z=0$, and hence $\frac{1}{z}\cdot\frac{\partial u_1}{\partial z}$ has at most a simple pole at $z=0$.  Therefore,
\[\{\bm{F},z\}=\dfrac{1-{\alpha_j}^2}{2z^2}+\frac{d_j}{z}+\psi_j(z),\]
for some constant $d_j$ and $\psi_j$ holomorphic.
\end{proof}
\subsection{Preparation Theorem for Local Expressions of Developing Maps}
\begin{thm}
Let $\bm{F}: M\setminus\{P_1, \cdot\cdot\cdot, P_n\}\rightarrow\mathbb{C}$ be a projective multi-valued locally univalent meromorphic function, and the monodromy of $\bm{F}$ belong to $Iso(\mathbb{C},\abs{dz}^2)$.  If $\bm{F}$ is compatible with the divisor $D=\prod_jP_j^{(\alpha_j-1)}$, then there exists a neighborhood $U_j$ of $P_j$ with complex coordinate $x$ and $T_j\in PGL(2,\Complex)$ such that $x(P_j)=0$ and $g_j=T_j\circ\bm{F}$ has the form
\begin{eqnarray*}
  g_j(x)&=&x^{\alpha_j}\,\,\,\,\,if\,\,\alpha_j\neq0 \\
  g_j(x)&=&\log{x}\,\,\,\,\,if\,\,\alpha_j=0.
\end{eqnarray*}
\end{thm}
\begin{proof}
By the proposition about the Schwarzian $\{\bm{F},x\}$, there is a neighborhood $U_j$ of $P_j$ and a local coordinate $x$ on $U_j$ such that $x(P_j)$ and
\[\{\bm{F},x\}=\dfrac{1-{\alpha_j}^2}{2x^2}+\frac{d_j}{x}+\psi_j(x),\]
where $\psi_j(x)$ is holomorphic near $P_j$.  Moreover, $d_j$ and $\psi_j(x)$ depend on the choice of the local coordinate $x$.  By the proposition in \cite{Yo87}, $\bm{F}$ has the form
\[\bm{F}=\dfrac{u_0(x)}{u_1(x)},\]
where $u_0(x)$ and $u_1(x)$ are two linearly independent solutions of the equation
\[\dfrac{d^2u}{dx^2}+\frac{1}{2}\biggl(\dfrac{1-{\alpha_j}^2}{2x^2}+\frac{d_j}{x}+\psi_j(x)\biggl)u=0.\]

Moreover, when $\bm{F}$ changes projectively, i.e.
\[\bm{F}\mapsto\dfrac{a\bm{F}+b}{c\bm{F}+d},\]
where $ad-bc=1$.  Then $u_0$ and $u_1$ would change as
\[\left(
  \begin{array}{c}
    u_0 \\
    u_1 \\
  \end{array}
\right)\mapsto
\left(
  \begin{array}{cc}
    d & c \\
    b & a \\
  \end{array}
\right)
\left(
  \begin{array}{c}
    u_0 \\
    u_1 \\
  \end{array}
\right),\]
and vice versa.

Note that the monodromy of $\bm{F}$ is in $Iso(\Complex, \abs{dz}^2)$, thus the solutions $u_0$ and $u_1$ could only transfer under the transformation of the form
\[\left(
    \begin{array}{cc}
      1 & 0 \\
      w & e^{i\theta} \\
    \end{array}
  \right),
\]
where $w\in \Complex$.
Now, we denote $L_j$ be the operator $x^2\frac{\partial^2}{\partial x^2}+q_j(x)$ with $q_j(x)=\frac{1}{2}(\frac{1-\alpha_j^2}{2}+d_jx+x^2\psi_j(x))$, and then $u_0$ and $u_1$ are both the solutions of the equation $L_ju=0$.  Moreover, the equation $L_ju=0$ has regular singularity at $0$, which allows us to quote the Frobenius method to solve it (cf \cite{Yo87}).

If
\[u(x)=x^s\sum_{j=0}^\infty c_jx^j,\,\,\,\,\,\,\,(c_0\equiv1)\]
is the solution of the equation $L_ju=0$ and
\[q_j(x)=\sum_{k=0}^{\infty} b_kx^k\]
be the power series expansion of $q_j(x)$ with $b_0=\frac{1-\alpha_j^2}{4}$.  Put
\[h(s)=s(s-1)+\dfrac{1-\alpha_j^2}{4},\]
and
\begin{eqnarray*}
  R_0 &=& 0, \\
  R_n &=& R_n(c_1, \cdot\cdot\cdot, c_{n-1},s)= \sum_{i=0}^{n-1}c_ib_{n-i},\,\,\,\,\,n>0.
\end{eqnarray*}
Then $L_ju=0$ if and only if we have the relations
\[h(s+n)c_n+R_n=0,\]
which we denote by $(\sharp)_n$, holds for all $n=0,1,2\ldots$.  Especially, the equation
\[h(s)=s(s-1)+\dfrac{1-\alpha_j^2}{4}=0\] is called the indicial equation of $L_ju=0$ at $x=0$.  This equation has two roots $s_1=\dfrac{1-\alpha_j}{2}$ and $s_2=\dfrac{1+\alpha_j}{2}$ with $s_2-s_1=\alpha_j$.

Now, $c_j$ is determined by the relations $(\sharp)_n$ and the parameter $s$, so we write $u(x)$ as
\[u(s,x)=x^s\sum_{j=0}^\infty c_j(s)x^j,\,\,\,\,\,\,\,(c_0\equiv1).\]
Suppose $s_{i_1}$ $(i_1=1\,\,or\,\,2)$ is the bigger one between $s_1$ and $s_2$.  So $h(s_{i_1}+n)\neq0$ for all $n\geq1$, then $c_n$ can be solved from $(\sharp)_n$ for every $n\geq 0$ and $u(s_{i_1},x)$ is a solution of the equation.  When consider the other root $s_{i_2}$ $(i_2\neq i_1\,\,\, and \,\,\,i_2=1\,\,or\,\,2)$, we need sperate into the following two cases.

\emph{Case 1}  If $s_2-s_1=\alpha_j$ is not an integer, then $h(s_{i_2}+n)\neq0$ for all $n\geq1$.  Thus, $u(s_{i_2}, x)$ is a solution linearly independent of $u(s_{i_1}, x)$.  Summing up, we have
\[u(s_{i_1}, x)=x^{s_{i_1}}(1+\varphi_{i_1}(x))\,\,\,\,\,and\,\,\,\,\,u(s_{i_2}, x)=x^{s_{i_2}}(1+\varphi_{i_2}(x)),\]
where $\varphi_{i_1}$ and $\varphi_{i_2}$ are holomorphic functions vanishing at $0$.  Then $u_0$ and $u_1$ are linear combinations of $u(s_{i_1}, x)$ and $u(s_{i_2}, x)$.  Therefore, $\bm{F}=\frac{u_0}{u_1}$ is a fractional transformation of $\frac{u(s_{i_1}, x)}{u(s_{i_2}, x)}$.  And hence $\bm{F}$ has the form
\[\bm{F}=\dfrac{Ax^{\alpha_j}\phi_j+B}{Cx^{\alpha_j}\phi_j+D},\]
where $A,B,C,D\in\Complex$ satisfying $AD-BC\neq0$ and $\phi_j(x)$ equal to $\frac{1+\varphi_{i_1}(x)}{1+\varphi_{i_2}(x)}$ or its inverse.  It is clear that $\phi_j(x)$ is holomorphic with $\phi_j(0)=1$.  By changing coordinates, we can just write
\[\bm{F}=\dfrac{Ax^{\alpha_j}+B}{Cx^{\alpha_j}+D}.\]

\emph{Case 2}  Suppose that $m\triangleq \abs{s_2-s_1}=\abs{\alpha_j}$ is a nonzero integer.

\emph{Subcase 2.1} If $R_m=0$, then the equation $(\sharp)_n$ for $s=s_{i_2}$ for all $n\geq1$ could be solved by choosing $c_m$ arbitrarily.  And then obtain the solution $u(s_{i_2}, x)$ which is linearly independent of $u(s_{i_1}, x)$.  Therefore, we turn to Case 1.

\emph{Subcase 2.2} If $R_m\neq0$, put
\[u^*=x^{s_{i_2}}\sum_{k=0}^\infty c_k(s_{i_2})x^k,\]
where $c_0=1$, and $c_j$'s $(j\geq1\,\,and \,\,j\neq m)$ are determined by $(\sharp)_j$ and $c_m$ is chosen arbitrarily.  Then
\[U_0(x)\triangleq h'(s_{i_1})u^*-R_m\frac{\partial}{\partial s}u(s, x)\arrowvert_{s=s_{i_1}}\]
is a solution.  Note that $c_j$'s are holomorphic with respect to $s$ since $h$ and $R_n$ both make it.  Thus, the two linearly independent solutions of $L_ju=0$, say $\left(
      \begin{array}{c}
        U_0(x) \\
        u(s_{i_1}, x) \\
      \end{array}
    \right)$, is given by
\[
\left(
  \begin{array}{cc}
    x^{s_{i_2}} & -R_mx^{s_{i_1}}\log{x} \\
    0 & x^{s_{i_1}} \\
  \end{array}
\right)\cdot
\left(
  \begin{array}{c}
    h'(s_{i_1})\sum_{k=0}^{\infty} c_k(s_{i_2})x^k-R_mx^m\sum_{k=0}^{\infty}c_k'(s_{i_1})x^k \\
    \sum_{k=0}^{\infty}c_k(s_{i_1})x^k \\
  \end{array}
\right).
\]
After a local rotation around $x=0$, 
$\left(
      \begin{array}{c}
        U_0(x) \\
        u(s_{i_1}, x) \\
      \end{array}
    \right)$
turns to be 
\[
\left(
  \begin{array}{cc}
    e^{2\pi \sqrt{-1}s_{i_2}} & -R_me^{2\pi\sqrt{-1}s_{i_1}}2\pi\sqrt{-1} \\
    0 & e^{2\pi\sqrt{-1}s_{i_1}} \\
  \end{array}
\right)\cdot
\left(
  \begin{array}{c}
        U_0(x) \\
        u(s_{i_1}, x) \\
      \end{array}
    \right),
\]
equivalent to
\[e^{2\pi\sqrt{-1}s_{i_1}}
\left(
  \begin{array}{cc}
    1 & -R_m2\pi\sqrt{-1} \\
    0 & 1 \\
  \end{array}
\right)\cdot
\left(
  \begin{array}{c}
        U_0(x) \\
        u(s_{i_1}, x) \\
      \end{array}
    \right),
\]
(by direct calculation and using $s_{i_1}-s_{i_2}=m$ is an integer).

Since the linearly independent solutions $u_0$ are $u_1$ are linear combinations of $U_0(x)$ and $u(s_{i_1}, x)$, there is a fractional transformation such that
\begin{eqnarray*}
  \frac{a\bm{F}+b}{c\bm{F}+d} &=& \frac{U_0(x)}{u(s_{i_1},x)} \\
   &=& x^{s_{i_2}-s_{i_1}}\cdot\frac{\psi_{i_2}}{\psi_{i_1}}-R_m\log{x} \\
   &=& x^{-\abs{\alpha_j}}\cdot\psi-R_m\log{x},
\end{eqnarray*}
where 
\[\psi_{i_2}=h'(s_{i_1})\sum_{k=0}^{\infty} c_k(s_{i_2})x^k-R_mx^m\sum_{k=0}^{\infty}c_k'(s_{i_1})x^k,\]
\[\psi_{i_1}=\sum_{k=0}^{\infty}c_k(s_{i_1})x^k,\]
and $\psi=\frac{\psi_{i_2}}{\psi_{i_1}}$.

$ad-bc\neq0$ because $u_0$ and $u_1$ are linearly independent, and by multiplying a constant we could assume $ad-bc=1$.  Since $s_1\neq s_2$, we have $h'(s_{i_1})\neq0$, and then $\psi(0)\neq0$.

Now, \[\bm{F}=\frac{d(x^{-\abs{\alpha_j}}\psi-R_m\log{x})-b}{-c(x^{-\abs{\alpha_j}}\psi-R_m\log{x})+a}.\]
By the fact that the monodromy of $\bm{F}$ is in $Iso(\mathbb{C},\abs{dz}^2)$, we have
\[e^{i\theta_0}\frac{d(x^{-\abs{\alpha_j}}\psi+\log{x})-b}{-c(x^{-\abs{\alpha_j}}\psi+\log{x})+a}+w_0
=\frac{d(x^{-\abs{\alpha_j}}\psi+\log{x}+2\pi i)-b}{-c(x^{-\abs{\alpha_j}}\psi+\log{x}+2\pi i)+a}.\]
Expand this equation and let $x$ tend to zero, it turns out that
\[c(de^{i\theta_0}-cw_0)\psi(0)^2=cd\psi(0)^2.\]
Since $\psi(0)\neq0$, this is
\[c(de^{i\theta_0}-cw_0)=cd.\]
If $c=0$, then
\[\bm{F}=\frac{d}{a}(x^{-\abs{\alpha_j}}-R_m\log{x})-\frac{b}{a},\]
and
\[\abs{\bm{F}'(x)}=\abs{\frac{d}{a}(-\abs{\alpha_j}x^{-\abs{\alpha_j}-1}-\frac{R_m}{x})}.\]
Since the parabolic metric can be obtained by
\[g=\bm{F}_*(\abs{dz}^2)=\abs{\bm{F}'(x)}^2\abs{dx}^2,\]
where $z$ is the standard complex coordinate on $\Complex$, we have
\[g=\abs{\frac{d}{a}(-\abs{\alpha_j}x^{-\abs{\alpha_j}-1}-\frac{R_m}{x})}^2\abs{dx}^2.\]

By the definition of $g$, there is a local coordinate $\xi$ centered at $P_j$, under which $g$ has the form
\[g=\mu^2(\xi)\abs{\xi}^{2(\alpha_j-1)}\abs{d\xi}^2,\]
where $\mu^2(\xi)=e^{2u_1}$ is smooth with $\triangle u_1=0$ as we have proved.  However, since every harmonic function could be the real part of a holomorphic function, there exists a holomorphic function $V(\xi)$ such that $Re(V)=u_1$, and so $\mu^2(\xi)=e^{2u_1}=\abs{e^{2V}}$.  Now, we can write
\begin{eqnarray*}
  g &=& \abs{\frac{d}{a}(-\abs{\alpha_j}x(\xi)^{-\abs{\alpha_j}-1}-\frac{R_m}{x(\xi)})}^2\abs{x'(\xi)}^2\abs{d\xi}^2 \\
    &=& \mu^2(\xi)\abs{\xi}^{2(\alpha_j-1)}\abs{d\xi}^2,
\end{eqnarray*}
where $x(\xi)$ is also a biholomorphic changing of coordinates and $\mu(\xi)\neq0$.  And $x$ tends to zero as $\xi$ tends to zero and vice versa.  Moreover, $x'(\xi)\neq0$ since $x(\xi)$ is biholomorphic.
This equation equals to
\[\abs{\dfrac{x'(\xi)(\abs{\alpha_j}x(\xi)^{-\abs{\alpha_j}-1}+R_mx(\xi)^{-1})}{e^{V(\xi)}\xi^{\alpha_j-1}}}^2=\abs{\frac{a}{d}}.\]
We denote $Q(\xi)$ as
\[Q(\xi)=\dfrac{x'(\xi)(\abs{\alpha_j}x(\xi)^{-\abs{\alpha_j}-1}+R_mx(\xi)^{-1})}{e^{V(\xi)}\xi^{\alpha_j-1}}.\]
Since it is holomorphic in a sector having vertex at $\xi=0$, $Q(\xi)$ is an open map if it is not a constant map.  Then the norm of $Q(\xi)$ cannot be constant $\frac{a}{d}$ on an open neighborhood in this sector if it is not a constant map.  Thus, $Q(\xi)=\frac{a}{d}e^{i\bar{\theta}}$ for some $\bar{\theta}$ fixed in one sector.  Moreover, on the overlapping parts of different sectors, $Q(\xi)$ must be the same constant.  So, then, in the punctured disk centered at $\xi=0$,
\[\frac{d}{a}\cdot\frac{\xi}{x}\cdot\dfrac{x'(\xi)(\abs{\alpha_j}+R_mx(\xi)^{\abs{\alpha_j}})}{e^{V(\xi)}e^{i\bar{\theta}}}=\xi^{\alpha_j}\cdot x^{\abs{\alpha_j}}.\]
This impossible for $\alpha_j>0$ since $\xi^{\alpha_j}\cdot x^{\alpha_j}=0$ at $x=\xi=0$ but the left hand side is 
\[\frac{d}{a}\alpha_j\cdot e^{-V(0)}e^{-i\bar{\theta}},\]
which is nonzero.

If $\alpha_j<0$, consider
\[\omega(x)=(\abs{\alpha_j}x^{-\abs{\alpha_j}-1}+\frac{R_m}{x})dx\]
and
\[\omega(\xi)=\frac{a}{d}e^{V(\xi)}e^{i\bar{\theta}}\xi^{\alpha_j-1}d\xi,\]
the same meromorphic differential under different coordinate charts in a punctured disk.  However, this cannot be true since they have different residues at $x=\xi=0$ for 
\[res_{x=0}\omega(x)=2\pi\sqrt{-1}R_m\neq 0,\]
and
\[res_{\xi=0}\omega(\xi)=0.\]

Therefore, $c\neq0$.  And we have
\[de^{i\theta_0}-d=cw_0.\]

From the local monodromy of $\frac{U_0(x)}{u(s_{i_1},x)}$, the fractional transformation of local monodromy for $\bm{F}$ is given by the matrix
\[e^{2\pi\sqrt{-1}s_{i_1}}
\left(
\begin{array}{ccc}
 d & -b \\

 -c & a
\end{array}
\right)
\left(
\begin{array}{ccc}
 1 & -2R_m\pi\sqrt{-1} \\

 0 & 1
\end{array}
\right),
\]
which is
\[e^{2\pi\sqrt{-1}s_{i_1}}
\left(
\begin{array}{ccc}
 d & -b-dR_m2\pi\sqrt{-1} \\

 -c & a+cR_m2\pi\sqrt{-1}
\end{array}
\right).
\]
And the local monodromy is also given by 
\[\left(
\begin{array}{ccc}
 e^{i\theta_0} & w_0 \\

 0 & 1
\end{array}
\right).\]
Since these two matrices give the same fractional transformation, then
\[ce^{i\theta_0}=0.\]
By $c\neq0$, $e^{i\theta_0}=0$, a contradiction.  That is, we rule out Subcase 2.2.






\emph{Case 3}  If $s_1=s_2$, then $s_1$ is a double root of $h$, and thus $h'(s_1)=0$.  Since $L_j$ does not depend on $s$ and so $L_j$ and $\frac{\partial}{\partial s}$ are commutative, the following expression
\[\frac{\partial}{\partial s}u(s,x)\biggl{\arrowvert}_{s=s_1}=u(s_1,x)\log{x}+x^{s_1}\sum_{j\geq0}c_j'(s_1)x^j\]
is a solution of $L_ju=0$.
Therefore,
\[\dfrac{a\bm{F}+b}{c\bm{F}+d}=\dfrac{\frac{\partial}{\partial s}u(s,x)\arrowvert_{s=s_1}}{u(s_1,x)}.\]
And then,
\[\dfrac{a\bm{F}+b}{c\bm{F}+d}=\log{x}+\phi(x),\]
where $\phi(x)=\dfrac{\sum_{j\geq0}c_j'(s_1)x^j}{\sum_{j\geq0}c_j(s_1)x^j}$ is holomorphic near $x=0$ and $\phi(0)=0$ since $c_0=1$.  Changing the coordinate by $y=xe^{\phi(x)}$, we could also write as
\[\dfrac{a\bm{F}+b}{c\bm{F}+d}=\log{y}.\]
\end{proof}

\subsection{Proof of theorem 1.0.9}
\begin{proof}
Suppose first that $\alpha_j\neq0$.  By the local property of developing map, for every $P_j$, there is a neighborhood such that $g_j=T_j\circ\bm{F}$ has the form
\[g_j(x)=x^{\alpha_j},\]
where $T_j\in PGL(2,\mathbb{C})$ and $x$ is a local coordinate near $P_j$ with $z(P_j)=0$.  Thus, we have
\[\frac{a\bm{F}+b}{c\bm{F}+d}=x^{\alpha_j},\]
where
\[\begin{pmatrix}a &b \\ c & d\end{pmatrix} \in PGL(2,\mathbb{C}).\]
So, locally, $\bm{F}$ has the expression
\[\bm{F}(x)=\frac{dx^{\alpha_j}-b}{-cx^{\alpha_j}+a},\,\,\,\,\,\,ad-bc=1.\]
Then we have
\[\bm{F}'(x)=\dfrac{(ad-bc)\alpha_jx^{\alpha_j-1}}{(-cx^{\alpha_j}+a)^2}.\]
Since the monodromy of $\bm{F}$ belongs to $Iso(\mathbb{C},\abs{dz}^2)$, $x=r\cdot e^{i\theta}$,
\[\bm{F}(x)=\bm{F}(re^{i\theta})=\frac{d\cdot r^{\alpha_j}\cdot e^{i\theta\alpha_j}-b}{-c\cdot r^{\alpha_j}\cdot e^{i\theta\alpha_j}+a},\]
and
\begin{eqnarray*}
  \bm{F}(re^{i\theta}e^{2\pi i}) &=& \frac{d\cdot r^{\alpha_j}\cdot e^{i\theta\alpha_j}e^{2\pi i\alpha_j}-b}{-c\cdot r^{\alpha_j}\cdot e^{i\theta\alpha_j}e^{2\pi i\alpha_j}+a} \\
   &=& \mathscr{T}(\bm{F}(re^{i\theta})) \\
   &=& e^{i\theta_0}\bm{F}(re^{i\theta})+w_0 \\
   &=& e^{i\theta_0}\frac{d\cdot r^{\alpha_j}\cdot e^{i\theta\alpha_j}-b}{-c\cdot r^{\alpha_j}\cdot e^{i\theta\alpha_j}+a}+w_0,
\end{eqnarray*}
where $\mathscr{T}\in Iso(\mathbb{C},\abs{dz}^2)$.

Reducing the above equality, we have
\[-cde^{2\pi i\alpha_j}x^{2\alpha_j}+(ade^{2\pi i\alpha_j}+bc)x^{\alpha_j}-ab\]
\[=-ce^{2\pi i\alpha_j}(de^{i\theta_0}-cw_0)x^{2\alpha_j}+[a(de^{i\theta_0}-cw_0)-ce^{2\pi i\alpha_j}(aw_0-be^{i\theta_0})]x^{\alpha_j}+a(aw_0-be^{i\theta_0}).\]
This equation holds only in a neighborhood of $x=0$.  However, by identity theorem of holomorphic functions, this equation holds in a whole univalent branch of $x^{\alpha_j}$.

Therefore, there are equalities
\begin{eqnarray*}
  cde^{2\pi i\alpha_j}    &=& ce^{2\pi i\alpha_j}(de^{i\theta_0}-cw_0); \,\,\,\,\,\,\,\,\,\,\,\,\,\,\,\,\,(1)\\
  ade^{2\pi i\alpha_j}+bc &=& a(de^{i\theta_0}-cw_0)-ce^{2\pi i\alpha_j}(aw_0-be^{i\theta_0}); \,\,\,\,\,\,\,\,(2)\\
  -ab                     &=& a(aw_0-be^{i\theta_0}). \,\,\,\,\,\,\,\,\,\,\,\,\,\,\,\,(3)
\end{eqnarray*}

If $c\neq0$, (1) turns out to be
\[d=de^{i\theta_0}-cw_0.\,\,\,\,\,\,\,\,\,\,\,\,(4)\]
By (2), we divide this equation to two cases.
\begin{itemize}
\item{\emph{Case 1: $a=0$.} \\By (2),\[bc=bce^{2\pi i\alpha_j}e^{i\theta_0}\]\[b=be^{2\pi i\alpha_j}e^{i\theta_0}.\]\\Since $ad-bc\neq0$ and $a=0$, we have $b\neq0$.   So, it comes to \[e^{2\pi i\alpha_j}e^{i\theta_0}=1.\]}
\item{\emph{Case 2: $a\neq0$.}\\By (3), \[a\omega_0=be^{i\theta_0}-b.\]  Substitute into (2), it turns that \[ade^{2\pi i\alpha_j}+bc=ade^{i\theta_0}-c(be^{i\theta_0}-b)+bce^{2\pi i\alpha_j},\] which is \[(ad-bc)e^{2\pi i\alpha_j}=(ad-bc)e^{i\theta_0}.\] So,\[e^{2\pi i\alpha_j}=e^{i\theta_0}.\]  Moreover, by (4) and $ad-bc\neq0$,\[ad=ade^{i\theta_0}-c(be^{i\theta_0}-b),\]\[e^{i\theta_0}=1.\] Thus, we have \[e^{2\pi i\alpha_j}=e^{i\theta_0}=1.\]}
\end{itemize}

Now, under the condition of monodromy and the above discussion, $\bm{F}$ has expression:
\begin{itemize}
\item{\emph{Case 1: $a=0$.}\[\bm{F}(x)=-\frac{d}{c}+\frac{b}{cx^{\alpha_j}}.\]}
\item{\emph{Case 2: $c=0$.}\[\bm{F}(x)=\frac{d}{a}x^{\alpha_j}-\frac{b}{a}.\]}
\item{\emph{Case 3: $a\neq0$, $c\neq0$ and $\alpha_j\in\mathbb{Z}$.}\[\bm{F}(x)=\frac{dx^{\alpha_j}-b}{-cx^{\alpha_j}+a},\] where $ad-bc\neq0.$}
\end{itemize}
In Case 1, locally, the conical metric $g$ has form
\begin{eqnarray*}
  g &=& \abs{d\bm{F}(x)}^2=\biggl|\frac{b^2\alpha_j^2}{c^2}\biggl|\cdot\abs{x}^{-2(\alpha_j+1)}\abs{dx}^2 \\
    &=& \mu^2(\xi)\abs{\xi}^{2(\alpha_j-1)}\abs{d\xi}^2,
\end{eqnarray*}
where the second formula is the definition of $g$.  And as before, we have
\[\frac{b\alpha_j}{c}x^{-(\alpha_j+1)}dx=e^{i\tilde{\theta}}e^{V(\xi)}\xi^{(\alpha_j-1)}d\xi,\]
for some $\tilde{\theta}$ fixed.

Then
\[\frac{b\alpha_j}{c}x^{-\alpha_j}\frac{dx}{d\xi}=e^{i\tilde{\theta}}e^{V(\xi)}\xi^{\alpha_j}\frac{x}{\xi}.\]
Since $\frac{dx}{d\xi}\neq0$, $\mu(0)\neq0$ and $\frac{x}{\xi}$ tends to $x'(0)\neq0$ as $\xi$ tends to zero, the left and right hand side tend to zero and infinity respectively as $\xi$ tends to zero, a contradiction.  

In Case 3, for $\alpha_j\in \mathbb{Z}$ and $\alpha_j>0$, locally, $\bm{F}$ could be expressed as
\begin{eqnarray*}
  \bm{F} &=& \frac{dx^{\alpha_j}-b}{-cx^{\alpha_j}+a} \\
         &=& -\frac{d}{c}+\frac{ad-bc}{ac}\cdot\frac{1}{1-\frac{c}{a}x^{\alpha_j}}  \\
         &=& -\frac{d}{c}+\frac{ad-bc}{ac}(1+\frac{c}{a}x^{\alpha_j}+(\frac{c}{a})^2x^{2\alpha_j}+\cdot\cdot\cdot) \\
         &=& C_0+C_1x^{\alpha_j}+C_2x^{2\alpha_j}+\cdot\cdot\cdot,
\end{eqnarray*}
where $C_1\not=0$.  Thus,
\[\bm{F}=C_0+x^{\alpha_j}g(x)\]
with $g$ holomorphic and $g(0)\not=0$.

If $\alpha_j\in \mathbb{Z}$ and $\alpha_j<0$, by the definition of the metric $g$,
\begin{eqnarray*}
  g &=& \abs{d\bm{F}(x)}^2=\abs{\dfrac{(ad-bc)\alpha_jx^{\alpha_j-1}}{(-cx^{\alpha_j}+a)^2}}^2\cdot\abs{dx}^2 \\
    &=& \mu^2(\xi)\abs{\xi}^{2(\alpha_j-1)}\abs{d\xi}^2.
\end{eqnarray*}
Thus,
\[\abs{(ad-bc)\alpha_j}^2\cdot\abs{\frac{x}{\xi}}^{2(\alpha_j-1)}\cdot\abs{\frac{dx}{d\xi}}^2=\mu^2(\xi)\abs{-cx^{\alpha_j}+a}^2.\]
Since $ad-bc\neq0$, $\alpha_j\neq0$, $\frac{x}{\xi}$ tends to $\frac{dx}{d\xi}(0)\neq0$ as $\xi\rightarrow 0$ and $\frac{dx}{d\xi}\neq0$ for every point in its domain, the coefficient $c$ must be zero in the right hand, otherwise the right hand of above equality blows up as $\xi$ tends to zero while the left hand does not.  Hence, for $\alpha_j\in \mathbb{Z}$ and $\alpha_j<0$,
\[\bm{F}(x)=\frac{d}{a}x^{\alpha_j}-\frac{b}{a}.\]

Summing up, we could conclude that the developing map could always be expressed as
\[\bm{F}=C_0x^{\alpha_j}+C_1,\]
where $\alpha_j\neq0$ and $C_0$, $C_1$ are constants.

Now, we turn to consider the case of $\alpha_j=0$.  We denote $\log{x}$ by $J(x)$.  Then
\[\bm{F}=\dfrac{dJ(x)-b}{-cJ(x)+a};\]
\[\bm{F}'(x)=\dfrac{(ad-bc)J'}{(a-cJ)^2}.\]
Thus, we have
\begin{eqnarray*}
  g &=& \abs{\bm{F}'(x)}^2\abs{dx}^2 \\
    &=& \dfrac{\abs{ad-bc}^2\abs{\frac{1}{x}}^2}{\abs{a-cJ}^4}\abs{dx}^2 \\
    &=& \mu^2(\xi)\frac{1}{\abs{\xi}^2}\abs{d\xi}^2
\end{eqnarray*}
as in the theorem .  And then
\[\dfrac{(ad-bc)}{x(a-cJ)^2}dx=e^{V(\xi)}e^{i\theta}\frac{1}{\xi}d\xi\]
for some $\theta$ fixed and this equality equals to
\begin{eqnarray*}
  (ad-bc)\frac{\xi}{x}\frac{dx}{d\xi} &=& e^{V(\xi)}e^{i\theta}(a-cJ)^2 \\
                                                    &=& e^{V(\xi)}e^{i\theta}(a-c\log{x})^2.
\end{eqnarray*}
However, this equality cannot be true if $c\neq0$ since the right hand side tends to infinity as $\xi$ tends to zero while the left hand side is finite.  So $c=0$.

Therefore,
\[\bm{F}(x)=\frac{d}{a}\log{x}-\frac{b}{a}.\]
\end{proof}

\section{Proof of Theorem 1.0.2}
\begin{proof}
First, suppose such a metric $g$ exists.  Let $\bm{F}$ be a developing map of $g$.  By the theorem about local expression of $\bm{F}$, we know that $d\bm{F}$ is a regular multiplicative differential whose divisor is $D=\prod_jP_j^{(\alpha_j-1)}$ and having normalized character.  However, from our theorem, this Prym differential exists uniquely up to multiplying nonzero constants.  Since $\bm{F}$ is obtained by integrating this Prym differential, we could conclude that $\bm{F}$ is unique up to transformations in $Iso(\Complex, \abs{dz}^2)$.  By the definition of developing maps, the parabolic conical metric on $M$ exists uniquely.

Next, the existence of the parabolic metrics representing the divisor $D$ could be obtained from the existence of such Prym differential having $D$ as its divisor.  For the divisor $D$ satisfying Gauss-Bonnet condition, we know that there exists uniquely a Prym differential $\bm{\omega}$ whose divisor is $D$.  Suppose that we could expressed $\bm{\omega}$ in the punctured disk centered at $P_j$ as
\[\bm{\omega}=f(x)dx,\]
where $x$ is a local coordinate centered at $P_j$ on $M$.  By the definition of regularity, $P_j$ is a simple pole of $\frac{df}{f}$.
Then
\[\frac{df}{f}=\frac{\alpha_j-1}{x}+t_j(x),\]
where $t_j(x)$ is holomorphic near $P_j$.   Thus, we have
\[f(x)=e^{\int t_j(x)dx}\cdot x^{\alpha_j-1}.\]
So
\[\abs{f(x)}^2\abs{dx}^2=\abs{e^{2\int t_j(x)dx}}\abs{x}^{2(\alpha_j-1)}\abs{dx}^2\]
defines a parabolic conical metric having angle $2\pi\alpha_j$ at $P_j$ and its developing maps are $\int^z \bm{\omega}$.  Therefore, we prove the sufficiency of Gauss-Bonnet condition.

Finally, since $d\bm{F}$ is a Prym differential, we have
\[deg(\bm{F})=\chi(M),\]
which is stated as a corollary in the page 129 of \cite{riesurf}.  Thus, we have proved the necessity of Gauss-Bonnet condition.
\end{proof}

\begin{example}
Now, we give a concrete expression of parabolic metrics on the two-sphere $\mathbb{S}^2$.  Let $D=\prod_jP_j^{(\alpha_j-1)}$ be a divisor with real coefficients on $\mathbb{S}^2$ with
\[\sum_{j=1}^n\alpha_j-1=\chi(\mathbb{S}^2)=-2.\]
Take Prym differential
\[\bm{\omega}=\prod_{j=1}^n (z-P_j)^{\alpha_j-1}dz,\]
where $z$ is the standard coordinate on $\mathbb{S}^2=\mathbb{C}\cup\{\infty\}$.  By our theorem above, $g=\abs{\bm{\omega}}^2$ is the unique parabolic metric representing the divisor $D$ on $\mathbb{S}^2$.
\end{example}
We note that $\bm{\omega}$ here is the same differential as in the well-know Schwarz-Christoffel formula and the developing map of $g$, say, $\int^z \bm{\omega}$, generalize the Schwarz-Christoffel formula to arbitrary points $P_j$'s and arbitrary real numbers $\alpha_j$'s with $\sum\alpha_j=n-2$.


School for the Gifted Young, University of Science and Technology of China, Hefei, China

E-mail address:santaiqu@gmail.com
\end{document}